\pdfoutput=1
\RequirePackage{ifpdf}
\ifpdf 
\documentclass[pdftex]{sigma}
\else
\documentclass{sigma}
\fi

\numberwithin{equation}{section}

\newtheorem{Theorem}{Theorem}[section]
\newtheorem*{Theorem*}{Theorem}

\newtheorem{Lemma}[Theorem]{Lemma}

 { \theoremstyle{definition}
\newtheorem{Definition}[Theorem]{Definition}

\newtheorem{Remark}[Theorem]{Remark} }

\newcommand{\C}{\mathbb{C}}

\newcommand{\eps}{\varepsilon}
\renewcommand{\epsilon}{\varepsilon}
\newcommand{\e}{\mathrm{e}}
\newcommand{\ii}{\mathrm{i}}

\newcommand{\norm}[2][]{{\left\|#2\right\|}_{#1}}

\renewcommand{\phi}{\varphi}
\newcommand{\R}{\mathbb{R}}
\newcommand{\sclp}[2][]{{\left\langle#2\right\rangle} _{#1}}

\DeclareMathOperator{\dom}{dom}

\begin{document}

\newcommand{\arXivNumber}{2109.05465}

\renewcommand{\thefootnote}{}

\renewcommand{\PaperNumber}{105}

\FirstPageHeading

\ShortArticleName{A Sharp Lieb--Thirring Inequality for Functional Difference Operators}

\ArticleName{A Sharp Lieb--Thirring Inequality\\ for Functional Difference Operators\footnote{This paper is a~contribution to the Special Issue on Mathematics of Integrable Systems: Classical and Quantum in honor of Leon Takhtajan.

~~\,The full collection is available at \href{https://www.emis.de/journals/SIGMA/Takhtajan.html}{https://www.emis.de/journals/SIGMA/Takhtajan.html}}}

\Author{Ari LAPTEV~$^{\rm ab}$ and Lukas SCHIMMER~$^{\rm c}$}

\AuthorNameForHeading{A.~Laptev and L.~Schimmer}

\Address{$^{\rm a)}$~Department of Mathematics, Imperial College London, London SW7 2AZ, UK}
\EmailD{\href{mailto:a.laptev@imperial.ac.uk}{a.laptev@imperial.ac.uk}}
\Address{$^{\rm b)}$~Saint Petersburg State University, Saint Petersburg, Russia}

\Address{$^{\rm c)}$~Institut Mittag--Leffler, The Royal Swedish Academy of Sciences, 182 60 Djursholm, Sweden}
\EmailD{\href{mailto:lukas.schimmer@kva.se}{lukas.schimmer@kva.se}}

\ArticleDates{Received September 12, 2021, in final form November 25, 2021; Published online December 06, 2021}

\Abstract{We prove sharp Lieb--Thirring type inequalities for the eigenvalues of a class of one-dimensional functional difference operators associated to mirror curves. We furthermore prove that the bottom of the essential spectrum of these operators is a resonance state.}


\Classification{47A75; 81Q10}

\begin{flushright}
\begin{minipage}{73mm}
\it To our friend and coauthor Leon Takhtajan\\ on the occasion of his 70th birthday
\end{minipage}
\end{flushright}

\renewcommand{\thefootnote}{\arabic{footnote}}
\setcounter{footnote}{0}

\section{Introduction}

Let $P$ be the self-adjoint quantum mechanical momentum operator on $L^2(\R)$, i.e., $P=-\ii\frac{\mathrm{d}}{\mathrm{d}x}$ and for $b>0$ denote by $U(b)$ the Weyl operator $U(b)=\exp(-bP)$. By using the Fourier transform
\begin{gather*}
\widehat{\psi}(k)=(\mathcal{F}\psi)(k)=\int_{\R}\e^{-2\pi\ii kx}\psi(x)\,{\rm d} x
\end{gather*}
we can write the domain of $U(b)$ as
\begin{gather*}
\dom(U(b))=\big\{\psi\in L^2(\R)\colon \e^{-2\pi b k}\widehat{\psi}(k)\in L^2(\R)\big\}.
\end{gather*}
Equivalently, $\dom(U(b))$ consists of those functions $\psi(x)$ which admit an analytic continuation to the strip $\{z = x+\ii y\in\C\colon 0<y <b\}$ such that $\psi(x +\ii y) \in L^2(\R)$ for all $0\leq y < b$ and there is a limit $\psi(x +\ii b - \ii 0) = \lim_{\eps\to 0^+}\psi (x + \ii b - \ii\eps)$ in the sense of convergence in $L^2(\R)$, which we will denote simply by $\psi(x+\ii b)$. The domain of the inverse operator $U(b)^{-1}$ can be characterised similarly.

For $b>0$ we define the operator $W_0(b)=U(b)+U(b)^{-1}=2\cosh(bP)$ on the domain
\begin{gather*}
\dom(W_0(b))=\big\{\psi\in L^2(\R)\colon 2\cosh(2\pi bk)\widehat{\psi}(k)\in L^2(\R)\big\}.
\end{gather*}
The operator $W_0(b)$ is self-adjoint and unitarily equivalent to the multiplication operator \linebreak $2\cosh(2\pi b k)$ in Fourier space. Its spectrum is thus absolutely continuous covering the interval
$[2,\infty)$ doubly.

Let $V\ge0$, $V\in L^1(\R)$ now be a real-valued potential function. The scalar inequality $2\cosh(2\pi b k)-2\ge (2\pi b k)^2$ implies the operator inequality
\begin{gather}
W_0(b)-2\ge-b^2\frac{\mathrm{d}^2}{\mathrm{d}x^2}
\label{WH}
\end{gather}
on $\dom(W_0(b))$. By Sobolev's inequality, we can conclude that the operator
\begin{gather*}
W_V(b)=W_0(b)- V
\end{gather*}
is symmetric and bounded from below on the common domain of $W_0(b)$ and $V$.
We can thus consider its Friedrichs extension, which we continue to denote by $W_V(b)$. This operator acts as
\begin{gather*}
(W_V(b)\psi)(x)=\psi(x+\ii b)+\psi(x-\ii b)-V(x)\psi(x).
\end{gather*}
Furthermore, by an application of Weyl's theorem (in a version for quadratic forms) and Rellich's lemma together with the fact that the form domain of $W_0(b)$ is continuously embedded in $H^1(\R)$ (as discussed at the beginning of Section~\ref{sec:sharp}) the spectrum of $W_V(b)$ consists of essential spectrum $[2,\infty)$ and discrete finite-multiplicity eigenvalues below. Details of this argument in the similar case of a Schr\"odinger operator can be found in the upcoming book \cite[Proposition 4.14]{FLW}.

We will show that the discrete spectrum satisfies a version of Lieb--Thirring inequalities for $1/2$-Riesz means.
When formulating the main result of the paper it is convenient to parametrise the eigenvalues (repeated with multiplicities) as $\lambda_j = - 2\cos(\omega_j)$, where $\omega_j\in [0,\pi]$ for
$\lambda_j\in[-2,2]$ and $\omega_j\in \ii [0,\infty)$ for $\lambda_j\le -2$.
Note that in the latter case $\lambda_j = - 2\cosh(|\omega_j|)$.

\begin{Theorem}\label{1}
Let $V\ge 0$ and let $V\in L^1(\mathbb R)$. If $W_V(b)\ge-2$, then
the discrete eigenvalues $\lambda_j=-2\cos(\omega_j)\in [-2,2)$ $($repeated with multiplicities$)$ satisfy
\begin{gather}\label{main}
\sum_{j\ge1}\frac{\sin\omega_j}{\omega_j}\le \frac{1}{2\pi b}\int_{\R}V(x) \, {\rm d} x.
\end{gather}
The constant in the inequality \eqref{main} is sharp in the sense that there is a potential
$V$ such that~\eqref{main} becomes equality.
\end{Theorem}

\begin{Remark}
Note that Theorem \ref{1} does not allow to estimate eigenvalues below $-2$. In~fact, from the proof of this theorem, the case of one eigenvalue below $-2$ could be included in the inequality \eqref{main}. We expect that the inequality holds true for all eigenvalues below $-2$. However, the method we use in the proof prevents us from including all eigenvalues due to oscillating properties of the resolvent $(W_0(b) - \lambda)^{-1}$ for $\lambda<-2$.
\end{Remark}

Lieb--Thirring inequalities were first established for Schr\"odinger operators in \cite{LT}. For a~one-dimensional Schr\"odinger operator ${-}\frac{\mathrm{d}}{\mathrm{d}x^2}-V$ on $L^2(\R)$ with negative eigenvalues \mbox{$\mu_1\le\mu_2\allowbreak\le\!\cdots\!<0,\!$} these bounds state that for any $\gamma\ge1/2$ there is a constant $L_\gamma>0$ such that
\begin{gather}
\sum_{j\ge 1}|\mu_j|^{\gamma}\le L_{\gamma}\int_{\R}V(x)^{\gamma+1/2}\, {\rm d} x
\label{LTSch}
\end{gather}
for all $V\ge0$, $V\in L^{\gamma+1/2}(\R)$. The condition $\gamma\ge1/2$ is optimal. Inequality \eqref{WH} implies that
\begin{gather}
\sum_{j\ge 1}|\lambda_j-2|^{\gamma}\le \frac{L_{\gamma}}{b}\int_{\R}V(x)^{\gamma+1/2}\, {\rm d} x
\label{LTS}
\end{gather}
for all eigenvalues $\lambda_j\le 2$ of $W_V(b)$. Under the additional assumption $W_V(b)\ge-2$, our bound~\eqref{main} presents an improvement of \eqref{LTS} for $\gamma=1/2$. This can be seen from the fact that for $\gamma=1/2$ the sharp constant in \eqref{LTSch} is given by $L_{1/2}=1/2$ \cite{HLT} and from the strict inequality
\begin{gather*}
|\lambda_j-2|^\frac12=|2\cos\omega_j+2|^\frac12<\frac{\pi\sin\omega_j}{\omega_j}
\end{gather*}
for $\omega_j\in[0,\pi)$.
The difference of the terms above vanishes as $\omega_j\to\pi$, implying that \eqref{LTS} is asymptotically optimal for small coupling. While the necessity of $\gamma\ge1/2$ in the Lieb--Thirring inequality for Schr\"odinger operators does not allow us to conclude that \eqref{LTS} fails for $0\le\gamma<1/2$, we will prove the following.

\begin{Theorem}\label{th:res}
Let $b>0$. If $V\in L^1(\R)$ with $\int_{\R}V\, {\rm d} x>0$, then $W_V(b)$ has at least one eigenvalue below $2$. Furthermore, if $0\le\gamma<1/2$, then there is no constant $L_\gamma$ such that \eqref{LTS} holds for all compactly supported $V$. This conclusion holds even under the assumption that $W_V(b)\ge-2$.
\end{Theorem}

The study of different properties of functional difference operators $W_V(b)$ was considered before.
In the case when $-V =V_0= \e^{2\pi b x}$ is an exponential function, the operator $W_{V_0}(b)$ first appeared in the study of the quantum Liouville model on the lattice \cite{FT1} and plays an important role in the representation theory of the non-compact quantum group $\mathrm{SL}_{q}(2,\R)$. The spectral analysis of this operator was first studied in \cite{K}, see also
\cite{FT2}.
In the case when $-V = 2\cosh (2\pi b x)$ the spectrum of $W_{V}(b)$ is discrete and converges to $+\infty$. Its Weyl asymptotics were obtained in \cite{LST1}. This result was extended to a class of growing potentials in \cite{LST2}. More information on spectral properties of functional difference operators can be found in papers \cite{GHM,GKMR,KM,KMZ, T}.

The proof method of Theorem \ref{1} is similar to the proof of the sharp Lieb--Thirring inequa\-lity~\eqref{LTSch} for a one-dimensional Schr\"odinger operator in the case $\gamma=1/2$ as presented in \cite{HLW}. It~relies on a property of convolutions of the resolvent kernels of the operator under consideration. Such a semigroup property was also recently established for Jacobi operators where it was again used to prove sharp Lieb--Thirring type inequalities \cite{LLS}. With a different proof (not using the convolution property) the sharp inequalities for the Schr\"odinger operator and the Jacobi operator were first obtained in \cite{HLT} and in \cite{HS}, respectively. Despite formal similarity to the case of Jacobi operators, it is still surprising that the proof method works for functional difference operators $W_V(b)$. These operators could be considered as differential operators of infinite order since the symbol $\cosh(2\pi b k)$ can be written as an infinite Taylor series of symbols of even degree w.r.t.\ the variable $k$.

\section{Free resolvent}

Since $W_0(b)\ge 2$ we conclude that $W_0(b)-\lambda$ is an invertible operator for $\lambda<2$. Let
$\lambda=-2\cos(\omega)$ with $\omega\in[0,\pi]$ if $\lambda\in[-2,2]$ and $\omega\in\ii[0,\infty)$ if $\lambda<-2$. Then in Fourier space the inverse of $W_0(b)-\lambda$ is given by the multiplication operator $(2\cosh(2\pi b k)+2\cos(\omega))^{-1}$.

Applying the inverse Fourier transform $\mathcal{F}^{-1}$ to $(2\cosh(2\pi b k)+2\cos(\omega))^{-1}$ we find the kernel of the free resolvent $G_\lambda=(W_0(b)-\lambda)^{-1}$ that is
\begin{gather}\label{rezolvent}
G_\lambda(x,y) = G_\lambda(x-y)=\frac{1}{2b\sin \omega} \frac{\sinh \big(\frac{\omega}{b} (x-y)\big)}{\sinh \big(\frac{\pi}{b} (x-y)\big)}.
\end{gather}

\begin{Remark}
Note that $G_\lambda(x-y)$ is an even and positive kernel for $\omega\in[0,\pi]$ and it becomes oscillating if
$\omega\in \ii(0,\infty)$.
This fact is one of the reasons why we are able to study Lieb--Thirring inequalities only for the eigenvalues $\lambda_j \in [-2,2]$. This interval contains all of the discrete spectrum if the potential $V$ is ``small'' enough. However, if $V$ generates eigenvalues lying in $(-\infty,-2)$, then the oscillating property of the Green's function prevents us from obtaining the desired inequality for all eigenvalues.
\end{Remark}

Note that the value of $G_\lambda$ on the diagonal $x=y$ takes the form
\begin{gather}\label{G0}
G_\lambda(0)=\frac{1}{2\pi b}\frac{\omega}{\sin\omega}
\end{gather}
and we can see the relation between the right-hand side of \eqref{G0} and the expression in the left-hand side of \eqref{main}.
Due to our parameterisation of the spectral parameter, the convergence $\lambda\to 2^-$ implies $\omega\to \pi^-$ and thus
\begin{gather*}
G_\lambda(0) \sim \frac{1}{2b} \frac{1}{\sqrt{1-\cos^2\omega}} \sim \frac{1}{2b} \frac{1}{\sqrt{2-\lambda}}
\qquad {\rm as} \quad \lambda\to 2^-.
\end{gather*}
If $\lambda\to-\infty$, then $\omega \to \ii\infty$ and
\begin{gather*}
G_\lambda(0) \sim \frac{1}{\pi b} |\lambda|^{-1} \log |\lambda|.
\end{gather*}
In \cite{FT2} L.~Faddeev and L.A.~Takhtajan studied the resolvent in a slightly different form
\begin{gather*}
G_\lambda(x,y) = \frac{\sigma}{\sinh \big(\frac{\pi \ii \varkappa}{\sigma}\big) } \bigg( \frac{\e^{-2\pi \ii \varkappa(x-y)}}{1-\e^{- 4 \pi \ii \sigma (x-y)}} +
\frac{\e^{2\pi \ii \varkappa(x-y)}}{1-\e^{4 \pi \ii \sigma (x-y)}} \bigg),
\end{gather*}
which coincides with \eqref{rezolvent} with $\sigma = \ii/2b$, $\lambda = 2\cosh(2b\pi\varkappa)$ and $\varkappa = \frac{\omega - \pi}{2\pi \ii b}$.
It was pointed out that the free resolvent can be written using the analogues of the Jost solutions
\begin{gather*}
f_-(x,\varkappa) = \e^{-2\pi \ii \varkappa x} \qquad {\rm and} \qquad f_+(x,\varkappa) = \e^{2\pi \ii \varkappa x}
\end{gather*}
that appear in the theory of one-dimensional Schr\"odinger operators. Namely
\begin{gather*}
G_\lambda(x-y) = \frac{2\sigma}{C(f_-,f_+)(\varkappa)}
\bigg(\frac{f_-(x, \varkappa) f_+(y, \varkappa)}{1 - \e^{\frac{\pi \ii}{\sigma'} (x-y)}} +
\frac{f_-(y, \varkappa) f_+(x, \varkappa)}{1 - \e^{-\frac{\pi \ii}{\sigma'} (x-y)}} \bigg),
\end{gather*}
where $\sigma' \sigma = -1/4$ and where $C(f,g)$ is the so-called Casorati determinant (a difference analogue of the Wronskian) of the solutions of the functional-difference equation
\begin{gather*}
C(f,g)(x,\varkappa) = f(x+ 2\sigma', \varkappa)g(x,\varkappa) - f(x,\varkappa) g(x + 2\sigma', \varkappa).
\end{gather*}
For the Jost solutions $C(f_-,f_+)(x,\varkappa) = 2\sinh\big(\frac{\pi \ii\kappa}{\sigma}\big)$.

The equality $(W_0(b) - \lambda)G(x-y) = \delta(x-y)$ could be interpreted as an equation of distributions.
Since the functions $f_{\pm}(x,k)$ are Jost solutions, the distribution defined by $(W_0(b) - \lambda)\times G(x-y)$ is supported only at $x = y$, and its singular part coincides with the singular part of the distribution
\begin{gather*}
-\frac{2\sigma\sigma'}{\pi \ii C(f_-,f_+)(\varkappa)} \bigg(\frac{f_-(x+2\sigma',\varkappa) f_+(y,\varkappa)- f_-(y,\varkappa)f_+(x+2\sigma',\varkappa)}{x-y-\ii0}
\\ \hphantom{-\frac{2\sigma\sigma'}{\pi \ii C(f_-,f_+)(\varkappa)} \bigg(}
{} + \frac{f_-(x-2\sigma',\varkappa) f_+(y,\varkappa)
- f_-(y,\varkappa)f_+(x-2\sigma',\varkappa)}{x-y+\ii0}\bigg)
\end{gather*}
in the neighbourhood of $x = y$. This singular part is equal to
\begin{gather*}
-\frac{2\sigma\sigma'}{\pi \ii} \bigg(\frac{1}{x-y-\ii0} - \frac{1}{x-y+\ii0} \bigg) = \delta(x-y),
\end{gather*}
where the authors used the Sokhotski--Plemelj formula. This formula is similar to the respective formula for a Schr\"odinger operator when the Dirac $\delta$-function appears by differentiating a step function.

\section[Proof of inequality (1.2)]{Proof of inequality (\ref{main})}

\subsection{Some auxiliary results}
We first collect some results from \cite{HLW} verbatim.
Let $A$ be a compact operator on a Hilbert space~$\mathcal{G}$ and let us denote
\begin{gather*}
\|A\|_n=\sum_{j=1}^n\sqrt{\lambda_j(A^*A)},
\end{gather*}
where $\lambda_j(A^*A)$ are the eigenvalues of $A^*A$ in decreasing order.
Then by Ky Fan's inequality (see for example \cite[Lemma~4.2]{GK}) the functionals $\|\cdot\|_n$, $n = 1,2,\dots $,
are norms and thus for any unitary operator $Y$ in $\mathcal{G}$ we have
\begin{gather*}
\|Y^* A Y\|_n = \|A\|_n.
\end{gather*}

\begin{Definition}
Let $A$, $B$ be two compact operators on $\mathcal{G}$. We say that $A$ majorises $B$ or $B\prec A$, iff
\begin{gather*}
\|B\|_n \le \|A\|_n, \qquad \text{for all} \quad n \in \mathbb N.
\end{gather*}
\end{Definition}

\begin{Lemma}\label{lem:maj}
Let $A$ be a nonnegative compact operator acting in $\mathcal{G}$, $\{Y(k)\}_{k\in \mathbb R}$ be a family of unitary operators on $\mathcal{G}$, and let $g(k)\, {\rm d} k$ be a probability measure on $\mathbb R$. Then the operator
\begin{gather*}
B = \int_{\mathbb R} Y(k)^* A Y(k) g(k) \, {\rm d} k
\end{gather*}
is majorised by~$A$.
\end{Lemma}

\begin{proof}This is a simple consequence of the triangle inequality
\begin{gather*}
\|B\|_n \le \int_{\mathbb R} \| Y^*(k) A Y(k)\|_n g(k) \, {\rm d} k = \|A\|_n \int_{\mathbb R} g(k) \, {\rm d} k = \|A\|_n.\tag*{\qed}
\end{gather*}
 \renewcommand{\qed}{}
\end{proof}

Let $\lambda_j=-2\cos\omega_j\le 2$ be the eigenvalues of $W_0(b)-V$ with $V\ge0$. In order to slightly simplify the notations it is convenient to write
\begin{gather*}
\lambda_j=-2\cos\big(\sqrt{\theta_j}\big)
\end{gather*}
with $\theta_j\in\big({-}\infty,\pi^2\big]$ and $\omega_j^2=\theta_j$.

Let us denote by $K_\lambda$ the Birman--Schwinger operator
\begin{gather}\label{BSch}
K_\lambda = V^{1/2} G_\lambda V^{1/2}.
\end{gather}
Let $\mu_j(K_\lambda)$ be the eigenvalues (in decreasing order) of the Birman--Schwinger operator $K_\lambda$ defined in \eqref{BSch}.
Then due to the Birman--Schwinger principle we have
\begin{gather}\label{BSch1}
1 = \mu_j(K_{\lambda_j}).
\end{gather}
Let us define the operator
\begin{gather*}
L_{\theta}:=\frac{1}{G_{-2\cos\sqrt{\theta}}(0)} K_{-2\cos\sqrt{\theta}} ,
\end{gather*}
where $G_{-2\cos\sqrt{\theta}}(0) =\frac{1}{2\pi b}\frac{\sqrt{\theta}}{\sin\sqrt{\theta}}$ is given in \eqref{G0}.
Then from \eqref{BSch1} we obtain
\begin{gather*}
\sum_{j\ge 1} \frac{1}{G_{\lambda_j}(0)}
=\sum_{j\ge 1} \frac{1}{G_{\lambda_j}(0)} \mu_j(K_{\lambda_j})=\sum_{j\ge 1}\mu_j(L_{\theta_j}).
\end{gather*}
The integral kernel of the operator $L_\theta$ is given by $\sqrt{V(x)}g_{\pi^2,\theta}(x-y)\sqrt{V(y)}$, where
\begin{gather*}
g_{\pi^2,\theta}(x):=
\frac{\pi}{\sqrt{\theta}}\frac{\sinh \big(\frac{\sqrt{\theta}}{b}x\big)}{\sinh \big(\frac{\pi}{b}x\big)}.
\end{gather*}
Consider a more general function
\begin{gather*}
g_{\varphi,\theta}(x):=
\frac{\sqrt\varphi}{\sqrt{\theta}}\frac{\sinh \big(\frac{\sqrt{\theta}}{b}x\big)}{\sinh \big(\frac{\sqrt\varphi}{b}x\big)}.
\end{gather*}
Since $g_{\varphi,\theta}(0)=1$ its Fourier transform
$\widehat{g}_{\varphi,\theta} = \mathcal{F} (g_{\varphi,\theta})$ satisfies the equation
\begin{gather*}
\int_{\mathbb R} \widehat{g}_{\varphi,\theta}(k) \, {\rm d} k = 1.
\end{gather*}
Moreover, for any $-\infty<\theta <\varphi$ with $0<\varphi <\pi^2$ we have
\begin{gather*}
 \widehat{g}_{\varphi,\theta}(k) = \mathcal{F} \Bigg(\frac{\sqrt\varphi}{\sqrt{\theta}} \frac{\sinh\big(\frac{\sqrt{\theta}}{b}x\big)}{\sinh\big(\frac{\sqrt\varphi}{b}x\big)}\Bigg)(k)
 = \frac{2\pi \sin\big(\pi \frac{\sqrt\theta}{\sqrt\varphi}\big)}{\sqrt\theta} \frac{b}{2\cosh \big(\frac{2\pi^2 b k}{\sqrt\varphi}\big) + 2\cos\big(\frac{\pi\sqrt\theta}{\sqrt\varphi}\big)},
\end{gather*}
and the right-hand side is positive. Thus $\widehat{g}_{\varphi,\theta}\, {\rm d} k$ is a probability measure for such values.

Note also that importantly
\begin{gather*}
\frac{g_{\pi^2,\theta}(x)}{g_{\pi^2,\theta'}(x)}
=\frac{\sqrt{\theta'}}{\sqrt{\theta}}\frac{\sinh\big(\frac{\sqrt{\theta}}{b}x\big)} {\sinh\big(\frac{\sqrt{\theta'}}{b}x\big)}=g_{\theta',\theta}(x)
\end{gather*}
and therefore
\begin{gather*}
\big(\widehat{g}_{\pi^2,\theta'}*\widehat{g}_{\theta',\theta}\big)(k)=\widehat{g}_{\pi^2,\theta}(k).
\end{gather*}
This is the interesting convolution/semigroup property mentioned in the introduction. In the special case $-\infty<\theta <0=\theta'$ analogous computations lead to the same result with $\widehat{g}_{0,\theta}(k)=\chi_{[-1,1]}\big(2\pi bk/\sqrt{|\theta|}\big)\pi b/\sqrt{|\theta|}$.

\begin{Lemma}[monotonicity]\label{Mon}
For $(\theta, \theta')$ such that $-\infty<\theta \le \theta'$ and $0\le\theta' <\pi^2$ we have $L_\theta\prec L_{\theta'}$.
\end{Lemma}

\begin{proof}
Let $Y(k)\colon L^2(\R)\to L^2(\R)$ be the unitary multiplication operator
\begin{gather*}
(Y(k)\psi)(x)=\e^{-2\pi \ii k x}\psi(x)
\end{gather*}
and let $T$ be the projection onto $V^{1/2}$, i.e.,
\begin{gather*}
(T\psi)(x)=V^{1/2} (x) \int_\R V^{1/2}(y) \psi(y) \, {\rm d} y.
\end{gather*}
Using $Y(k'+k'')=Y(k')Y(k'')$ and Lemma \ref{lem:maj} we obtain
\begin{align*}
L_{\theta}&=\int_\R Y(k)^*T Y(k) \widehat{g}_{\pi^2,\theta}(k)\, {\rm d} k
\\
&=\int_\R\int_{\R} Y(k)^*T Y(k) \widehat{g}_{\pi^2,\theta'}(k') \widehat{g}_{\theta',\theta}(k-k')\, {\rm d} k' {\rm d} k
\\
&=\int_\R Y(k'')^*\bigg(\int_{\R} Y(k')^*T Y(k') \widehat{g}_{\pi^2,\theta'}(k')\, {\rm d} k'\bigg)Y(k'') \widehat{g}_{\theta',\theta}(k'')\, {\rm d} k'' \prec L_{\theta'},
\end{align*}
where we have used that $\widehat{g}_{\theta',\theta} \, {\rm d} k$ is a probability measure.
\end{proof}

\begin{Remark}
With a slight abuse of notations, Lemma \ref{Mon} says that $L_{\lambda}\prec L_{\lambda'}$ for any $\lambda<2$ as long as $\lambda\le \lambda'$ and $-2\le\lambda' <2$.
\end{Remark}

\subsection[Proof of inequality (1.2)]{Proof of inequality (\ref{main})}

We now enumerate the eigenvalues of the operator $W_V(b)$ belonging to the interval $[-2,2)$ such that
$-2\le \lambda_1\le\lambda_2\le\lambda_3\le \cdots$ repeated with multiplicity.
By using the monotonicity established in Lemma~\ref{Mon} we have a sequence of inequalities
\begin{gather*}
\frac{1}{G_{\lambda_1}(0)} =2\pi b \frac{\sin\omega_1}{\omega_1} = \mu_1(L_{\theta_1}) \le \mu_1(L_{\theta_2}),
\\
\sum_{j=1}^2 \frac{1}{G_{\lambda_j}(0)} = 2\pi b \sum_{j=1}^2 \frac{\sin\omega_j}{\omega_j}
\le \sum_{j=1}^2 \mu_j(L_{\theta_2})
\le \sum_{j=1}^2 \mu_j(L_{\theta_3}),
\\
\sum_{j=1}^3 \frac{1}{G_{\lambda_j}(0)} = 2\pi b \sum_{j=1}^3 \frac{\sin\omega_j}{\omega_j}
\le \sum_{j=1}^3 \mu_j(L_{\theta_3})
\le \sum_{j=1}^3 \mu_j(L_{\theta_4}), \qquad {\rm etc}.
\end{gather*}
Note that we do not use any assumptions on the multiplicities of the eigenvalues, other than their finiteness. Furthermore, by Lemma \ref{Mon} the same results also hold true if a single eigenvalue is below $-2$.
Continuing the above process and noting that the trace of $L_{\theta}$ is $\int_\R V \, {\rm d} x$ for all $\theta$, we finally obtain
\begin{gather*}
\sum_{j\ge1}\frac{\sin\omega_j}{\omega_j}\le \frac{1}{2\pi b}\int_{\R}V(x)\, {\rm d} x.
\end{gather*}
The proof is complete.

\begin{Remark}
Note that $\frac{2\cosh(2\pi bk)-2}{b^2}\to(2\pi k)^2$ tends to the symbol of the second derivative as $b\to0$ and that $W_{b^2V}(b)\ge-2$ for sufficiently small $b$. We thus expect that it should be possible to recover the Lieb--Thirring inequality \eqref{LTSch} for a Schr\"odinger operator with the sharp constant $L_{1/2}=1/2$ from Theorem \ref{1}.
\end{Remark}

\section[Sharpness of inequality (1.2)]{Sharpness of inequality (\ref{main})}\label{sec:sharp}

Similarly to the case of Schr\"odinger operators, we aim to prove that the Lieb--Thirring inequality becomes an equality for Dirac-delta potentials. To this end let $c>0$ and consider the potential $V_c(x)=c\delta(x)$. To properly define $W_{V_c}(b)$, we first note that the quadratic form $\sclp{\psi,(W_0(b)-2)\psi}$ can be written as
\begin{gather}
\sclp{\psi,(W_0(b)-2)\psi}=\int_{\R}\big|2\sinh(\pi b k)\widehat{\psi}(k)\big|^2 \, {\rm d} k=\int_{\R}|\psi(x+\ii b/2)-\psi(x-\ii b/2)|^2\, {\rm d} x.
\label{quadratic}
\end{gather}
This can be seen by introducing the self-adjoint operator $D(b)=U(b/2)-U(b/2)^{-1}= 2\sinh\big(\frac{bP}2\big)$ and checking that $D(b)^2=W_0(b)-2$ either directly or by means of the identity $\cosh(2\pi b k)-1\allowbreak=2\sinh(\pi b k)^2$. The form domain of $W_0(b)$ is thus $\dom(D(b))=\dom(W_0(b/2))\subset H^1(\R)$ and on this domain Sobolev's inequality yields that
\begin{align*}
|\psi(0)|^2&\le \eps\int_{\R}|\psi'(x)|^2 \, {\rm d} x+\frac1\eps\int_{\R}|\psi(x)|^2\, {\rm d} x\\
&\le\frac{\eps}{b^2}\int_{\R}\big|2\sinh(\pi b k)\widehat{\psi}(k)\big|^2\, {\rm d} k+\frac1\eps\int_{\R}|\psi(x)|^2\, {\rm d} x
\end{align*}
for any choice of $\eps>0$.
The KLMN theorem thus allows us to define $W_0(b)-V_c$.
As a rank one perturbation of the operator $W_0(b)$ the potential $V_c$ generates no more than one eigenvalue below the continuous spectrum $[2,\infty)$.

In Fourier space the eigenequation $(W_0(b)-c\delta)\psi_c=\lambda\psi_c$ becomes
\begin{gather*}
2\cosh(2\pi b k)\widehat{\psi_c}(k)-c \psi_c(0)=\lambda\widehat{\psi_c}(k)
\end{gather*}
by means of the formal identity $\mathcal{F}(\delta\psi_c)=\psi_c(0)$. Writing again $\lambda=-2\cos\omega$ we obtain
\begin{gather}\label{eigenf}
\widehat{\psi_c}(k)=\frac{c\psi_c(0)}{2\cosh(2\pi bk)+2\cos\omega}
\end{gather}
and therefore
\begin{gather}\label{psi0}
\psi_c(x)=c\psi_c(0) G_{-2\cos\omega}(x)
=\frac{c\psi_c(0)}{2b\sin\omega}\frac{\sinh\big(\frac{\omega}{b}x\big)}{\sinh\big(\frac{\pi}{b}x\big)}.
\end{gather}
Of course we could have seen this immediately by using the equation for the Green's function
\begin{gather*}
(W_0(b)+2\cos\omega)G_{-2\cos\omega}(x)=\delta(x).
\end{gather*}
Letting $x\to0$ in \eqref{psi0} we find
\begin{gather*}
1=\frac{c}{2b\sin\omega}\frac{\omega}{\pi}
\end{gather*}
or equivalently
\begin{gather}
\frac{\sin\omega}{\omega}=\frac{c}{2\pi b}.\label{eq:eig}
\end{gather}
Since $\frac{\sin\sqrt{\theta}}{\sqrt{\theta}}$ is a monotone decreasing function of $\theta=\omega^2\in\big({-}\infty,\pi^2\big]$ that takes all values in $[0,\infty)$, for any $c>0$ there is a unique solution $\omega_c$ to \eqref{eq:eig} and vice versa. If $c/(2\pi b)<1$ then $\omega_c\in(0,\pi)$ and otherwise $\omega_c\in\ii[0,\infty)$. Since $\int V_c\, {\rm d} x=c$, the identity \eqref{eq:eig} can be rewritten as
\begin{gather*}
\frac{\sin\omega}{\omega}=\frac{1}{2\pi b}\int_{\R}V_c(x) \, {\rm d} x
\end{gather*}
showing that the Lieb--Thirring inequality is satisfied for potentials $-c\delta$ with a single eigenvalue that can be placed anywhere in $(-\infty,2)$ by choosing $c>0$ suitably.

\begin{Remark}If we choose the normalising constant $\psi(0) >0$ then the eigenfunction defined in~\eqref{psi0}
\begin{gather*}
\psi_c(x)
=\frac{c\psi(0)}{2b\sin\omega_c}\frac{\sinh\big(\frac{\omega_c}{b}x\big)}{\sinh\big(\frac{\pi}{b}x\big)}
\end{gather*}
is positive assuming that the coupling constant $c$ is small enough satisfying the inequality $c/(2\pi b)\le1$ and thus $\omega_c \in[0,\pi)$. Note that if $c/(2\pi b) = 1$ then $\omega_c = 0$ and
\begin{gather*}
\psi_c(x) = \frac{\pi \psi(0) x}{b \sinh\big(\frac{\pi}{b}x\big)}>0.
\end{gather*}
However, if the coupling constant $c> 2\pi b$ then $\omega_c \in \ii(0,\infty)$ and hence
\begin{gather*}
\psi_c(x)=\frac{c\psi(0)}{2b\sinh|\omega_c|}\frac{\sin\big(\frac{|\omega_c|}{b}x\big)}{\sinh\big(\frac{\pi}{b}x\big)}
\end{gather*}
is an oscillating function and in particular has an infinite number of zeros. This contradicts a~possible conjecture that the eigenfunction for the lowest eigenvalue is strictly positive.
\end{Remark}

\noindent
{\bf Open problem.} Assume that the discrete spectrum $\sigma_{\rm d}(W_V(b))$ of the operator $W_V(b)$ satisfies the property $\sigma_{\rm d}(W_V(b))\subset [-2,2)$. Is it true that the eigenfunction corresponding to the lowest eigenvalue could be chosen strictly positive?

\section[Necessity of gamma >= 1/2]
{Necessity of $\boldsymbol{\gamma\ge1/2}$}

The following argument is similar to that presented in the upcoming book \cite[Propositions~4.41 and~4.42]{FLW} for the case of a Schr\"odinger operator.
For $\eps>0$ let $\psi_\eps(x)=1/\cosh(2\eps x/b)$. If $\eps$ is sufficiently small, say $\eps\le\eps_0$, then $\psi_\eps\in \dom(W_0(b))$. Using \eqref{quadratic} we compute that
\begin{gather}
\sclp{\psi_\eps,(W_0(b)-2)\psi_\eps}
=\frac{b\sin^2\eps}{2\eps}\int_{\R}\bigg|\frac{2\sinh x}{\cos^2\eps\cosh^2x+\sin^2\eps\sinh^2 x}\bigg|^2\, {\rm d} x
\le Cb\eps\label{Weps}
\end{gather}
for a constant $C>0$ independent of $\eps\le\eps_0$. For any potential $V\in L^1(\R)$ it holds that $\sclp{\psi_\eps,V\psi_\eps}\to\int_{\R}V\, {\rm d} x$ as $\eps\to0$ by dominated convergence and thus for sufficiently small $\eps$
\begin{gather*}
\sclp{\psi_\eps,(W_V(b)-2)\psi_\eps}<0.
\end{gather*}
By the min-max principle this proves the first part of Theorem \ref{th:res}.

For the second assertion of the theorem we choose more specifically the compactly supported potential $V(x)=c\chi_{[-1/2,1/2]}(x/b)$. By Sobolev's inequality $W_V(b)\ge -2$ for sufficiently small $c\le c_0$ such that all the discrete eigenvalues of $W_V(b)$ are contained in $[-2,2)$. Furthermore $\norm{\psi_\eps}^2=b/\eps$ and, since $\tanh x\ge x/2$ for $0\le x\le1$,
\begin{gather}
\sclp{\psi_\eps,V\psi_\eps}=cb\int_{-1/2}^{1/2}|\cosh(2\eps x)|^{-2}\, {\rm d} x=\frac{cb\tanh \eps}{\eps}\ge \frac12cb\label{Veps}
\end{gather}
for $\eps\le1$.
We now choose $\eps=c\delta$. If $\delta\le\min(\eps_0/c_0,1/c_0)$ such that $\eps\le\min(\eps_0,1)$, then \eqref{Weps} and \eqref{Veps} both hold and
\begin{gather*}
\frac{\sclp{\psi_\eps,(W_V(b)-2)\psi_\eps}}{\norm{\psi_\eps}^2}\le C\eps^2- \frac12c\eps=c^2\delta\bigg(C\delta-\frac{1}{2}\bigg).
\end{gather*}
Choosing $\delta<\min(\eps_0/c_0,1/c_0,1/2C)$ we can conclude by the min-max principle that $W_V(b)-2$ has a negative eigenvalue $\lambda_1\le-c^2\delta\big(\frac12-C\delta\big)$. If a Lieb--Thirring inequality \eqref{LTS} were to hold for $\gamma<1/2$ then for some finite $L_\gamma$
\begin{gather*}
c^{2\gamma}\delta^{\gamma}\bigg(\frac12-C\delta\bigg)^\gamma\le \frac{L_{\gamma}}{b}
\int_{\R} V(x)^{\gamma+\frac12}\, {\rm d} x=L_{\gamma}c^{\gamma+\frac12},
\end{gather*}
which is clearly a contradiction if $c\to0$.

\subsection*{Acknowledgements}
A.~Laptev was partially supported by RSF grant 18-11-0032. L.~Schimmer was supported by VR grant 2017-04736 at the Royal Swedish Academy of Sciences. The authors would like to thank the anonymous referees for their useful comments to improve the article.

\pdfbookmark[1]{References}{ref}
\LastPageEnding

\end{document}